\documentclass[12pt]{article}
\usepackage[centertags]{amsmath}     
\usepackage{amssymb}
\usepackage{amsthm}
\usepackage[alphabetic, initials]{amsrefs}
\usepackage{verbatim}                
\usepackage[dvips]{graphics}
\usepackage{enumerate}

\numberwithin{equation}{section}

\newtheorem{proposition}[equation]{Proposition}

\newtheorem{lemma}[equation]{Lemma}

\newtheorem{definition}[equation]{Definition}

\newtheorem{theorem}[equation]{Theorem}

\newtheorem*{thm1.3p}{Theorem 1.3'}
\newtheorem*{thm1.3pp}{Theorem 1.3''}

\newcommand{\eref}[1]{(\ref{#1})}

\newcommand{\eps}{\varepsilon}

\newcommand{\vphi}{\varphi}

\newcommand{\Cb}{\mathbb{C}}

\newcommand{\Nb}{\mathbb{N}}
\newcommand{\Rb}{\mathbb{R}}

\newcommand{\Fl}{\mathcal{F}}

\DeclareMathOperator{\Hess}{Hess}

\makeatletter
\def\blfootnote{\xdef\@thefnmark{}\@footnotetext}
\makeatother

\begin{document}
\title{A gradient estimate for harmonic functions sharing the same zeros}
\author{Dan Mangoubi}
\date{}
\maketitle

\begin{abstract}
Let $u, v$ be two harmonic functions in $\{|z|<2\}\subset\Cb$ 
which have exactly the same set $Z$ of zeros.
We observe that $\big|\nabla\log |u/v|\big|$ is bounded in the unit disk
by a constant which depends on $Z$ only. In case $Z=\emptyset$ this goes back
to Li-Yau's gradient estimate for positive harmonic functions.
The general boundary Harnack principle gives
 H\"older estimates on $\log |u/v|$.
\end{abstract}

\section{Introduction}
\subsection{Background and statement }
Consider a positive harmonic function $u$ 
in the disk of radius two
\mbox{$B_2\subset\Cb$}. The Harnack inequality gives a bound  on 
\mbox{$|\log u(z_1)-\log u(z_2)|$} 
where $z_1, z_2$ run in the unit disk.
This bound is independent of $u$.
If we let $v=1$ be the constant function then this is the same as 
saying that 
\begin{equation}
\left|\log\frac{u(z_1)}{v(z_1)}-\log \frac{u(z_2)}{v(z_2)}\right|
\end{equation}
is bounded by a constant independent on $u$ when 
$z_1, z_2\in B_1$. 

More generally, it is natural to ask what can be said about the quotient
$u/v$ in case $u, v$ are harmonic functions in $B_2$ which do change sign in $B_2$ and have exactly the same set of zeros.
The aim of this short note is to give an answer to this question in
two dimensions and to pose two related natural questions.

Let us assume for clarity of the introduction that
  $u=f\cdot v$ for some smooth function $f>0$ (in fact, below we show that this is always the case).
The boundary Harnack principle (BHP) (\cites{ancona-78, wu-78, caffarelli-fabes-mortola-salsa-81, jer-ken82, bal-vol96, pop-vol98})
applied to our situation shows that if $\Omega$ is a connected component of $\{u\neq 0\}$ then $\log u/v$ is a $C^\alpha$-function near $\partial\Omega\cap B_1$ for some $0<\alpha<1$.
 In this note we make the observation, maybe known to experts, that in fact we have a $C^1$-bound.
  Namely, we show that
$|\nabla\log \frac{u}{v}|$ is bounded in $B_1$. 
In the case $v$ is the constant function and $u$ is a positive
harmonic function this goes back to
Li-Yau's gradient estimate (\cite{li-yau-86}).

Another point of view of our observation is to say that
if in the BHP the harmonic functions can be extended across the boundary of the domain, then one has a $C^1$-bound on $\log u/v$.
In fact, an example due to Carlos Kenig (\S\ref{sec:kenig})
shows that one cannot in general obtain a $C^1$-bound
even  when the boundary of the domain
is composed of straight lines.
The precise result we show is
\begin{theorem}
\label{thm:main}
Let $Z\subset B_2$.
Let 
$$\Fl(Z)=\{ u\in C^2(B_2) |\ \Delta u=0 \mbox{ and } \{u=0\}=Z\} $$
Then for all $u, v\in \Fl(Z)$ $u/v$ extends to a smooth nowhere vanishing
function in $B_2$ and there exists a constant $C_Z>0$ such that 
$$\forall u, v\in\Fl(Z) \quad 
\big|\nabla\log |u/v|\big|\leq C_Z\mbox{ in } B_1.$$
\end{theorem}


\begin{subsection}{Two questions}
We pose two natural questions 
which arise naturally from Theorem~\ref{thm:main}.
\vspace{.8ex}

\noindent\textbf{Deformations of the zero set.}
A related result close in spirit to Theorem~\ref{thm:main} is proved in~\cite{nad99}. In this paper
Nadirashvili considers a variable-sign 
harmonic function in $B_2$, and proves the existence of a bound on 
$|u|$ in $B_1$ in terms of the first $3k$ derivatives of $u$ at $0$,
where $k$ is the number of connected components
of $B_2\setminus Z$. 
Nadirashvili's result 
hints that it may be true that the constant $C_Z$ in Theorem~\ref{thm:main} depends only on
the number of components of $B_2\setminus Z$.
Equivalently, we ask whether it is true that the constant $C_Z$
depends only on the number of intersection points of $Z$ with the
circle $\{|z|=3/2\}$.
\vspace{.7ex}

\noindent\textbf{Higher dimensions.}
We ask whether Theorem~\ref{thm:main} is a coincidence of dimension
two or stays true in dimension three.
It seems that even the case where $Z$ is the zero set of a quadratic harmonic polynomial in three variables is
 not known.
 \vspace{.8ex}
 
We hope to investigate both questions in the near future.
\end{subsection}

\subsection{Idea of proof} 
We give two proofs of Theorem~\ref{thm:main}.
Both proofs are built upon the cases where $Z$ is in normal form (say, $Z=\{\Im z^k=0\}$). This is the reason why our method
is restricted to two dimensions.

The first proof we give is based on the maximum principle 
in a similar spirit to~\cite{li-yau-86}.
 In a normal case we find a certain
positive definite quadratic form which leads to a Bochner type estimate.
In the case where $Z$ is empty (the Li-Yau case) the positivity of the quadratic form we define
is evident. Then, we use the Bochner type estimate to get a gradient estimate.

The second (shorter) proof, due to Misha Sodin,
is based on the Poisson formula for sectors of the plane.

We decided to keep both proofs in this note since the first
proof makes almost no use of explicit formulas, and
we hope it may be useful to extend Theorem~\ref{thm:main}
to situations where no explicit formulas exist.



\subsection{Organization of the paper}
The proof of the normal form case of Theorem~\ref{thm:main} 
is given in Section~\ref{sec:proof-normal-form}. The reduction
to the normal form case is done in Section~\ref{sec:proof}.
In Section~\ref{sec:sodin} we give an alternative proof
using Poisson formulas due to Misha Sodin. In sections~\ref{sec:formulas}-\ref{sec:singularity}
we develop different ingredients of the proof:
In Section~\ref{sec:formulas} we calculate the singular elliptic equation
satisfied by the quotient of two harmonic functions.
We also find a Bochner type formula.
In Section~\ref{sec:quadratic} we define and 
prove the positivity of a certain 
quadratic form needed in the proof of a Bochner type Inequality proved
in Section~\ref{sec:bochner}. In Section~\ref{sec:singularity}
we treat several expressions which involve logarithmic singularities.
Finally, in Section~\ref{sec:examples} we give several examples which illustrate Theorem~\ref{thm:main} and clarify it.
\vspace{1ex}

\subsection{Acknowledgements} I am happy to thank 
Jozef Dodziuk for  preliminary discussions 
on Li-Yau's gradient estimates.
I am grateful to David Kazhdan and Leonid Polterovich for
asking me  questions which led me to the present work.  
I thank Benji Weiss for an interesting example (\S\ref{sec:examples}).
I thank Gady Kozma and Fedja Nazarov
 for their interest in this work and for referring me
 to the BHP. I especially thank Misha Sodin for his encouragement and for finding a second simple proof~(\S\ref{sec:sodin}).
 I am grateful to  Carlos Kenig for his interest and
 for an illuminating example~(\S\ref{sec:kenig}) calrifying the relation of this note to the BHP.
 I am grateful to Charlie Fefferman for his advice and interest.
 I thank S.-T. Yau for his support.
  This research was supported by ISF grant 225/10 and by 
 BSF grant 2010214.
 
\section{A semi-linear singular elliptic equation}
\label{sec:formulas}
\begin{lemma}
\label{lem:h-pde}
Let $f>0$ be a smooth function. Let $v$ be harmonic.
If $fv$ is harmonic then
$h:=\log f$ satisfies the following second order quasilinear 
degenerate elliptic equation
\begin{equation}
\label{eqn:quasi-linear-degenerate-elliptic}
v\Delta h+2\langle \nabla v, \nabla h  \rangle 
+v|\nabla h|^2 =0\ .
\end{equation}
\end{lemma}
\begin{proof}
By the chain rule,
$$0 = \Delta(e^h v) 
= 2e^h\left(v\Delta h+\langle \nabla v, \nabla h\rangle 
+ v|\nabla h|^2\right)\ .$$
\end{proof}
\begin{lemma}
\label{lem:DeltaF}
Let $v$ be a smooth function and let $h$ satisfies 
equation~\eref{eqn:quasi-linear-degenerate-elliptic}.
Let $F=|\nabla h|^2$. Then
\begin{multline}
\label{eqn:DeltaF}
v^2\Delta F =2v^2|\Hess h|^2 +4\langle\nabla h, \nabla v\rangle^2
-4v(\Hess v)(\nabla h, \nabla h)\\-2v\langle\nabla F, \nabla v\rangle
-2v^2\langle\nabla F, \nabla h\rangle
\end{multline}
\end{lemma}
\begin{proof} 
It follows immediately from~\eref{eqn:quasi-linear-degenerate-elliptic}
that at points where $v=0$ equation~\eref{eqn:DeltaF} is also satisfied.
Therefore we can assume $v\neq 0$.
We differentiate
$$F_{,k} = 2\langle\nabla h, \nabla h_{,k}\rangle\quad\quad
F_{,kk} = 2\langle\nabla h_{,k},\nabla h_{,k}\rangle 
+ 2\langle\nabla h, \nabla h_{,kk}\rangle\ .$$
Hence, 
$$\Delta F = 2\sum_{k, l} |h_{,kl}|^2+2\langle \nabla h, \nabla \Delta h\rangle
=2|\Hess h|^2+2\langle\nabla h, \nabla \Delta h\rangle\ .$$
By equation~\eref{eqn:quasi-linear-degenerate-elliptic}
$$\langle\nabla h, \nabla \Delta h\rangle = -\langle \nabla h, \nabla \left(\frac{2\langle\nabla v, \nabla h\rangle}{v}\right) \rangle
-\langle \nabla h, \nabla F\rangle\ .$$
It remains only to simplify the expression on the right hand side: 
$$\nabla\left(\frac{2\langle\nabla v, \nabla h\rangle}{v}\right)=
-\frac{2\langle\nabla v, \nabla h\rangle}{v^2}\nabla v
+\frac{2\nabla(\langle\nabla v, \nabla h\rangle)}{v}\ ,$$
and an easy computation shows
$$
\langle\nabla h, 2\nabla(\langle\nabla v, \nabla h\rangle)\rangle = 2(\Hess v)(\nabla h, \nabla h)+\langle\nabla F, \nabla v\rangle\ .
$$
\end{proof}

%
%
\section{A non-negative quadratic form}
\label{sec:quadratic}
 Let $v_k(z)=\Im z^k$, where $k$ is a non-negative integer.
  In this section we define a quadratic
form related to $v_k$ and to equation~\eref{eqn:DeltaF}
and show it is non-negative. This will play a key role in
the proof of the Bochner type inequality in Lemma~\ref{lem:bochner}

\begin{definition}
Let $X$ be a smooth vector field on $B_2$. We define
 $$Q_k(X):= (X v_k)^2-v_k(\Hess v_k)(X, X)\ .$$
\end{definition}
\begin{proposition}
\label{prop:nonnegativity}
For all vector fields $X$ on $B_2$ the function 
$Q_k(X)$ is non-negative.
Moreover, we have
$$(Q_k)(X)\geq \frac{1}{k} (Xv_k)^2\ .$$
More precisely, we show
$$Q_k= \frac{1}{k}(dv_k)^2+k(k-1)r^{2k}(d\theta)^2\ .$$
\end{proposition}
\begin{proof}
Let $X, Y$ be two vector fields on $B_2$, and
let 
$$\tilde{B}(X, Y):=\frac{k-1}{k}(Xv_k)(Yv_k)-v_k\Hess(v_k)(X, Y)\ .$$
We compute in polar coordinates:
\begin{align*}
\partial_r v_k &= kr^{k-1}\sin k\theta ,&
\partial^2_{rr} v_k &=k(k-1)r^{k-2}\sin k\theta,\\
\partial_\theta v_k &= kr^k\cos k\theta, &
\partial^2_{\theta\theta} v_k &=-k^2r^{k}\sin k\theta \\
\partial^2_{r\theta} v_k &=k^2r^{k-1}\cos k\theta.
\end{align*}
The Levi-Civita connection is given by 
\begin{align*}
\nabla_{\partial_r}\partial_r &= 0, &
\nabla_{\partial_r}\partial_\theta 
&= \nabla_{\partial_\theta}\partial_r =\frac{\partial_\theta}{r}, &
\nabla_{\partial_\theta}\partial_\theta = -r\partial_r\ .
\end{align*}
%
%
Recall that $(\Hess f)(X, Y)=X(Yf) - (\nabla_X Y) f$. Hence,
\begin{align*}
\Hess(v_k)(\partial_r, \partial_r) &= k(k-1)r^{k-2}\sin k\theta, \\
\Hess(v_k)(\partial_r, \partial_\theta) &=k(k-1)r^{k-1}\cos k\theta, \\
\Hess(v_k)(\partial_\theta, \partial_\theta) &= -k(k-1)r^k\sin k\theta\ .
\end{align*} 
%
%
%
We see that
$$
 \tilde{B}(\partial_r, \partial_r)=\tilde{B}(\partial_r, \partial_\theta)=0,
\quad \tilde{B}(\partial_\theta, \partial_\theta)=k(k-1)r^{2k}\ .
$$
We conclude that
$$Q_k(X)-\frac{1}{k}(Xv_k)^2=\tilde{B}(X, X)=
k(k-1)r^{2k}d\theta(X)^2\geq 0.$$
\end{proof}

\section{An inequality \`a la Bochner}
\label{sec:bochner}
In this section we prove an inequality of Bochner's type.
This will be crucial to prove Theorem~\ref{thm:main}.

Let $v_k(z)=\Im z^k$ and let $h$ satisfy 
equation~\eref{eqn:quasi-linear-degenerate-elliptic} with $v=v_k$.
Let \mbox{$F=|\nabla h|^2$}. 
Then 
\begin{lemma} 
\label{lem:bochner}
$F$ satisfies
$$v_k^2\Delta F +2v_k\langle\nabla F, \nabla v_k\rangle
+2v_k^2\langle\nabla F, \nabla h\rangle \geq v_k^2 \frac{F^2}{k+1}$$
\end{lemma}
\begin{proof}
By Lemma~\ref{lem:DeltaF} and Proposition~\ref{prop:nonnegativity}
\begin{align*}
v_k^2\Delta F &+ 2v_k\langle\nabla F, \nabla v_k\rangle
+2v_k^2\langle\nabla F, \nabla h\rangle
\\&=
2v_k^2|\Hess h|^2+4\langle\nabla h, \nabla v_k\rangle^2-
4v_k\Hess (v_k)(\nabla h, \nabla h) \\ 
&\geq
2v_k^2|\Hess h|^2
+\frac{4}{k}\langle\nabla h, \nabla v_k\rangle^2
\geq v_k^2 (\Delta h)^2 +\frac{4}{k}\langle\nabla h, \nabla v_k\rangle^2
\\ &\geq
\frac{1}{k+1}\left(v_k\Delta h +2\langle \nabla v_k, \nabla h\rangle\right)^2
=v_k^2\frac{F^2}{k+1}\ ,
\end{align*}
where we used that $2|\Hess f|^2\geq (\Delta f)^2$ in $\Rb^2$,
 that $a^2+\frac{b^2}{k}\geq \frac{(a+b)^2}{k+1}$
for real numbers $a, b$ and equation~\eref{eqn:quasi-linear-degenerate-elliptic}.
\end{proof}
\section{Study of $\langle\nabla F, \nabla v_k\rangle/v_k$}
\label{sec:singularity}
We would like to show that 
$\frac{1}{v_k}\langle \nabla F, \nabla v_k\rangle$
extends to a smooth function in $B_2$.
We will apply the following standard division lemma:
\begin{lemma}
\label{lem:division-linear}
Let $l(x, y)=ax+by$. Let $f\in C^{\infty}(B_2)$ be such that
$f(x, y)=0$ whenever $l(x, y)=0$. Then
there exists $q\in C^{\infty}(B_2)$ such that
$f(x, y)=q(x, y)l(x, y)$.
\end{lemma}
\begin{proof}
We can assume without loss of generality that $b=0$ and $a\neq 0$.
Define
$$q(x, y)=\frac{1}{a}\int_0^1 (\partial_1 f)(tx, y)\,dt\ .$$
\end{proof}
\begin{lemma}
\label{lem:division-by-vk}
Let $f\in C^{\infty}(B_2)$ be such that $f=0$ whenever $v_k=0$.
Then, there exists $q\in C^{\infty}(B_2)$ such that $f=qv_k$.
\end{lemma}
\begin{proof}
$v_k$ can be expressed as a product of $k$ linear factors.
In fact
$$v_k(x, y)=a_k\prod_{l=0}^{k-1} 
\left(y\cos \frac{l\pi}{k}-x\sin \frac{l\pi}{k}\right)$$
%
for some $a_k\in\Rb$. Hence, the lemma follows by induction from Lemma~\ref{lem:division-linear}. 
\end{proof}
\begin{lemma}
\label{lem:nablav-nablaF}
There exists a smooth function $G_k$ such that 
$$\langle\nabla F, \nabla v_k\rangle =G_k v_k\ .$$
\end{lemma}
\begin{proof}
By Lemma~\ref{lem:division-by-vk}
it is enough to show that $\langle\nabla F, \nabla v_k\rangle$
vanishes whenever $v_k$ does. This will follow from 
equation~\eref{eqn:quasi-linear-degenerate-elliptic}. Indeed,
let $p\in B_2$, $p\neq 0$ be such that $v_k(p)=0$.
Then
$$\langle \nabla F, \nabla v_k\rangle(p)=F_{,r}(p)(v_k)_{,r}(p)+
\frac{1}{|p|^2}F_{,\theta}(p)v_{k,\theta}(p)=
\frac{1}{|p|^2}F_{,\theta}(p)v_{k,\theta}(p)\ .$$
We will show that $F_{,\theta}(p)=0$.
$$ F=h_{,r}^2+\frac{h_{,\theta}^2}{r^2}\ .$$
So, $$F_{,\theta}(p)=2h_{,r}(p)h_{,r\theta}(p) +\frac{2}{|p|^2}
h_{,\theta}(p)h_{,\theta\theta}(p)\ .$$
By equation~\eref{eqn:quasi-linear-degenerate-elliptic},
$$0=\langle\nabla h, \nabla v_k\rangle(p)
=h_{,r}(p)v_{k,r}(p)+\frac{1}{|p|^2}h_{,\theta}(p)v_{k,\theta}(p)
=\frac{1}{|p|^2}h_{,\theta}(p)v_{k,\theta}(p)\ .$$
Since $v_{k,\theta}(p)\neq 0$ we conclude that $h_{,\theta}(p)=0$.
Since, $h_{,\theta}(p)=0$ on the line passing through $0$ and $p$,
we also see that $h_{,r\theta}(p)=0$.
\end{proof}
The next lemma shows that the expression
$\langle\nabla F, \nabla v_k \rangle/v_k$ has the role of a second
derivative of $F$ on $v_k=0$. This will be important in the proof
of Theorem~\ref{thm:main}. 
\begin{lemma}
\label{lem:second-derivative}
Let $f, g\in C^{\infty}(B_2)$ be such that
$$\langle\nabla f, \nabla v_k \rangle = g v_k \ .$$
Let $p\in B_2$  be a local maximum point of $f$.
Then $g(p)\leq 0$. 
\end{lemma}
\begin{proof}
$$\langle\nabla f, \nabla v_k\rangle =
kr^{k-2}(rf_r \sin k\theta +f_{\theta}\cos k\theta)\ .$$
$$g(p)=\lim_{(r, \theta)\to(r_p, \theta_p)} k\frac{f_r}{r}
+k\frac{f_\theta\cos k\theta}{r^2\sin k\theta}\ .$$

Assume first that $p\neq 0$, then
$$\lim_{q\to p} k\frac{f_r(q)}{r_q} = 0$$
since $p$ is a critical point.
If $\sin k\theta_p\neq 0$ then
$$\lim_{q\to p} k\frac{f_\theta(q)\cos k\theta_q}{r_q^2\sin k\theta_q} = 0\ ,$$
since $p$ is a critical point.
Otherwise, by L'H\^opital's rule we have
$$\lim_{q\to p} k \frac{f_{\theta}(q)\cos k\theta_q}{r_q^2\sin k\theta_q }
=\lim_{\theta\to\theta_p} 
k\frac{f_{\theta\theta}(r_p, \theta)\cos k\theta-kf_{\theta}(q)
\sin k\theta}
{kr_p^2\cos k\theta}=\frac{f_{\theta\theta}(p)}{r_p^2}\leq 0\ ,$$
since $p$ is a maximum point.

If $p=0$, let $\theta_0 =\frac{\pi}{2k}$, 
then 
$$g(0)=\lim_{r\to 0^+} k\frac{f_r(r, \theta_0)}{r} \leq 0\ ,$$
since $0$ is a maximum point.
\end{proof}
\section{Proof of the Theorem~\ref{thm:main} in normal form}
\label{sec:proof-normal-form}
We prove Theorem~\ref{thm:main} for the case $Z=Z_k=\{\Im z^k=0\}\cap B_2$.
We first need to construct a suitable cutoff function.
Let $\chi\in C_c^{\infty}([0, 2))$ be such that
$\chi\equiv 1$ on $[0, 1]$ and $\chi$ is non-increasing.
Let $\vphi(x, y):=\chi(x^2+y^2)$.
Observe that 
\begin{equation}
\label{eqn:nablav-nablaphi}
\frac{\langle\nabla \vphi, \nabla v_k\rangle}{v_k} = 
2k\chi'(r^2)\ .
\end{equation}
Let $A>0$ be such that
\begin{equation}
\label{eqn:cutoff-bounds}
\chi'\geq -A ,\quad \Delta\vphi\geq -A, \quad |\nabla\vphi|^2\leq A\vphi.
\end{equation}
%
%

Without loss of generality we can assume $v=v_k$.
Let $u\in\Fl(Z_k)$.
\vspace{1ex}

\noindent\textbf{Existence of a positive quotient.}
We first show that $|u/v_k|$ defines a positive smooth function in $B_2$.
By Lemma~\ref{lem:division-by-vk} there exists
$f\in C^{\infty}(B_2)$ such that $u=fv_k$. It is clear
that $f\neq 0$ on $B_2\setminus Z_k$. Let $p\in Z_k$, $p\neq 0$.
If $f(p)=0$ then $u(p)=0$ and $(\nabla u)(p)=0$. 
Hence, $u$ has a zero
of order $d\geq 2$ at $p$, but 
since $u$ is harmonic in a small ball $B_{\eps}(p)$ 
centered at $p$ this implies that $Z\cap B_\eps(p)$
is homeomorphic to the zero set of $\Im z^d$. This is a contradiction.
If $p=0$ and $f(p)=0$ then $u(p)$ has a zero of order $d\geq k+1$ at $0$
with $Z_k$ as a zero set which is impossible since $u$ is harmonic.
\vspace{1ex}

\noindent\textbf{A gradient estimate.}
We now proceed to proving the gradient estimate on $\log f$.
Let $h=\log f$. 
$h\in C^{\infty}(B_2)$ and $u=e^h v_k$.
We let $F=|\nabla h|^2$. At points $q\in B_2$ where $v_k(q)\neq 0$
and $\vphi(q)\neq 0$ we have by Lemma~\ref{lem:bochner}
\begin{align}
\nonumber
\Delta(\vphi F) &=
(\Delta\vphi)F
+2\langle\nabla\vphi,\nabla F\rangle
+\vphi\Delta F\\
\nonumber
&\geq
(\Delta\vphi)F+2\langle\nabla\vphi,\nabla F\rangle
+\frac{\vphi F^2}{k+1}
-2\vphi\frac{\langle\nabla F, \nabla v_k\rangle}{v_k}
-2\vphi\langle\nabla F,\nabla h\rangle\\
\nonumber
&=(\Delta\vphi) F+\frac{2}{\vphi}\langle\nabla\vphi,\nabla(\vphi F)\rangle
-2F\frac{|\nabla \vphi|^2}{\vphi} +\frac{\vphi F^2}{k+1}
-2\frac{\langle\nabla (\vphi F), \nabla v_k\rangle}{v_k} \\
\nonumber
&+2F\frac{\langle\nabla \vphi, \nabla v_k\rangle}{v_k}
-2\langle\nabla (\vphi F), \nabla h\rangle
+2F\langle\nabla \vphi, \nabla h\rangle\\
\nonumber
&\geq 
(\Delta\vphi) F+\frac{2}{\vphi}\langle\nabla\vphi,\nabla(\vphi F)\rangle
-2F\frac{|\nabla \vphi|^2}{\vphi} +\frac{\vphi F^2}{k+1}
-2\frac{\langle\nabla (\vphi F), \nabla v_k\rangle}{v_k}\\
\label{eqn:bochner-phiF}
&+2F\frac{\langle\nabla \vphi, \nabla v_k\rangle}{v_k}
-2\langle\nabla (\vphi F), \nabla h\rangle
-2F^{3/2}|\nabla\vphi|
\end{align}
The last inequality follows from the Cauchy-Schwartz inequality.

Let $p\in B_2$ be a point where $\vphi F$ attains its maximum.
Observe that by Lemma~\ref{lem:nablav-nablaF} and~\eref{eqn:nablav-nablaphi}
there exists $G\in C_c^{\infty}(B_2)$ such that
\begin{equation}
\label{eqn:nablav-nablaphiF}
\langle\nabla (\vphi F), \nabla v_k\rangle=G v_k \ .
\end{equation}
It follows 
from~\eref{eqn:bochner-phiF},~\eref{eqn:nablav-nablaphiF}
and Lemma~\ref{lem:second-derivative} that at the point $p$ the following
inequality is satisfied 
$$0\geq (\Delta\vphi) F-2F \frac{|\nabla\vphi|^2}{\vphi}
+\frac{\vphi F^2}{k+1}
+4kF\chi'(|p|^2)-2F(\vphi F)^{1/2}\frac{|\nabla \vphi|}{\vphi^{1/2}}\ .
$$
Dividing by $F(p)$ and using~\eref{eqn:cutoff-bounds} we get
the following quadratic inequality in $(\vphi F)^{1/2}$:
$$0\geq \frac{\vphi F}{k+1}-2\sqrt{A}(\vphi F)^{1/2}-(4k+3)A\ ,$$
from which we conclude that $(\vphi F)^{1/2}(p)\leq 4(k+1)\sqrt{A}$.
Since $p$ is a maximum point, the same inequality is true for all
$q\in B_2$. In particular, in $B_1$ we get
$$|\nabla h|\leq 4(k+1)\sqrt{A}\ .$$
\qed
%
%

\section{Proof of Theorem~\ref{thm:main} - the general case}
\label{sec:proof}
In this case we reduce the general case to the $Z=Z_k$ proved in Section~\ref{sec:proof-normal-form}.
\begin{proof}
\textbf{Existence of a positive quotient.} Fix $v\in\Fl(Z)$.
Let $u\in\Fl(Z)$ be arbitrary.
We will first show that $|u/v|$ extends to a positive smooth
function in $B_2$. Let $p\in Z$. There exist $k(p)\in\Nb$, 
a neighborhood $N_p\ni p$, an injective conformal map
$\alpha_p:N_p\to B_1$ such that
$\alpha_p(p)=0$ 
and $v\circ\alpha_p^{-1}(w)=\Im w^{k(p)}$ for
all $w\in\alpha_p(N_p)$.
Let $N_p'$ be a neighborhood of $p$ such that
$\overline{N_p'}\subset N_p$. 
Let $W=\alpha_p(N_p)$ and $W'=\alpha_p(N_p')$.
$v\circ\alpha_p^{-1}$ and $u\circ\alpha_p^{-1}$ are harmonic functions
 both vanish  exactly on $Z_{k(p)}\cap W$.
By Section~\ref{sec:proof-normal-form} we know that $|(u\circ\alpha_p^{-1})/
(v\circ\alpha_p^{-1})|$
defines a positive smooth function $f_0$ on $W$. 
Let $f(z):= f_0(\alpha_p(z))$ be defined in $N_p$. 
Then $|f|>0$ and $u=fv$ in $N_p$. This shows that $|u/v|$ extends
to a smooth positive function in $N_p$. 
Since $p$ is arbitrary we conclude that $|u/v|$ extends to a smooth
positive function in~$B_2$.
\vspace{1ex}

\noindent\textbf{A bound on $\nabla\log|u/v|$.}
Let $p\in Z\cap \overline{B}_1$. 
Let $\alpha_p, k(p), N_p, N_p'$ be defined as above.
Let us write $u=e^{h}v$ where $h\in C^{\infty}(B_2)$.
By Section~\ref{sec:proof-normal-form} we know that 
$|\nabla (h\circ\alpha_p^{-1})|\leq C_p$ in $W'$.
By the chain rule it follows that
$|\nabla h|\leq C_p|\alpha_p'|$ in~$N_p'$. 
Since we can cover $Z\cap\overline{B}_1$ by a finite number of open
sets of the form $N_p'$ we get that 
\begin{equation}
\label{eqn:bound-nz}
|\ h|\leq C_v \mbox{ in a neighborhood } N_Z\mbox{ of } Z\cap \overline{B}_1.
\end{equation}
Since
$u, v$ do not vanish in a neighborhood of $\overline{B}_1\setminus N_Z$,  by~\cite{li-yau-86} (or by the case $Z=\emptyset$ in section~\ref{sec:proof-normal-form}) we know that 
\begin{equation}
\label{eqn:bound-minus-nz}
|\nabla h|\leq C \mbox{ in a neighborhood of } \overline{B}_1\setminus N_Z\ .
\end{equation}
From inequalities~\eref{eqn:bound-nz} and~\eref{eqn:bound-minus-nz}
we get that $|\nabla h|\leq C_Z$ in $\overline{B}_1$.
\end{proof}

\section{Examples}
\label{sec:examples}
\subsection{Harmonic functions sharing the same zeros}
We give a few examples of harmonic
functions with common zeros.
\begin{enumerate}
\item (Due to Benji Weiss) Let $\alpha\in[-\pi/2, \pi/2]$.
Let $u_{\alpha}(x, y)=e^{\alpha x}\sin|\alpha| y$.
 The zero set of $u_\alpha$ in $B_2$ is the
$x$-axis.

\item Let $F:B_2\to\Cb$ be holomorphic. Suppose $|F|<\pi$.
 The zero set of
$\Im e^F$ is the same as the zero set of $\Im F$. 

\item Let $\alpha\in\Rb$ be such that $0<|\alpha|\leq 1$.
Let $f(z)$ be the branch of $z^{\alpha}$ in $\Cb\setminus(-\infty, 0]$
which admits positive values on the positive real axis.
Define $u(z)=\Im f(z+2)$.
The zero set of $u$ in $B_2$ is the $x$-axis.
%
%

\item Let $F:B_2\to B_2$ be holomorphic.
Let $u=\Im \frac{aF}{cF+d}$ where $a, c, d\in\Rb$ are such that
 $F\neq -d/c$ in 
$B_2$. Then, the zero set of $u$ in $B_2$ coincides with the
zero set of $\Im F$.

\item 
Let $k\in\Nb$ and  let $S_k=\{0<\arg z<\pi/k\}\cap B_2$.
Let $u$ be a positive harmonic function in $S_k$, which is continuous on
$\partial S_k$ and vanishes on $\partial S_k\cap B_2$. One
can extend $u$ by reflections to a harmonic function in $B_2$.
The zero set of $u$ coincides with the zero set of $\Im z^k$.

\item (Due to Charlie Fefferman) Let $u(x, y)=xy$. Let $v(x, y)=x^3y-xy^3$.
Then $u$ and  $u+\eps v$ have the same zero set in $B_2$ 
if $\eps>0$ is sufficiently small.
\end{enumerate}
\subsection{An example clarifying the relation to the BHP}
\label{sec:kenig}
(Due to Carlos Kenig) Let $t>1$ and let $S=\{0<\arg z<2\pi/t\}\subset \Cb$.
Let $v=\Im z^{t/2}$. Observe that $v$ is positive in $S$ and 
$v|_{\partial S}=0$.
Let $p\in S$ be such that $|p|>2$. Let $G_p$ be the Green function of $S$ with singularity at $p$.
We consider $G_p/v$ in $S\cap B_2$. Unless $t$ is an integer, $G_p/v$ cannot
be extended as a $C^1$-function in a neighborhood of $0$.
Moreover, we note that 
$|\nabla \log (G_p/v)|$ is bounded in $S\cap B_1$ if and only if
$t\geq 2$.

A little simpler, we let
 $u=\Im z^{t/2}+\eps\Im z^t$ for small $\eps>0$.
 Then $\log u/v$ has no bounded gradient in $S\cap B_1$ unless
 $t\geq 2$.

These families of examples (for $1<t< 2$) show that the BHP alone
is not enough to obtain gradient estimates even if the boundary of the domain is nice (straight lines).

\section{A second proof of the normal form case 
using Poisson's formula}
\label{sec:sodin}
This section is due to Misha Sodin.
Let 
$$S_k=\{z\in\Cb |\, |z|<1,  0<\arg z<\pi/k\}\ .$$
Let $u$ be a positive harmonic function in $S_k$, continuous
on $\overline{S_k}$ such that $u=0$ on $\partial S_k\setminus \{|z|=1\}$.
We have the following integral representation for~$u$:

\begin{align*}
u(re^{i\theta}) &= \frac{k(1-r^{2k})}{2\pi}
\int_0^{\pi/k} \left(\frac{1}{|e^{ik\vphi}-r^ke^{ik\theta}|^2}
-\frac{1}{|e^{ik\vphi}-r^ke^{-ik\theta}|^2}\right)u(e^{i\vphi})\,d\vphi\\
&=
 \frac{2kr^{k}(1-r^{2k})\sin(k\theta)}{\pi}\int_0^{\pi/k}
\frac{\sin (k\vphi) u(e^{i\vphi})}
{|e^{ik\vphi}-r^ke^{ik\theta}|^2|e^{ik\vphi}-r^ke^{-ik\theta}|^2}\,d\vphi
\end{align*}

Hence

$$\frac{u(re^{i\theta})}{r^k\sin k\theta}
=\frac{2k(1-r^{2k})}{\pi}\int_{0}^{\pi/k}
\frac{\sin(k\vphi)u(e^{i\vphi})}{{|e^{ik\vphi}-r^ke^{ik\theta}|^2|e^{ik\vphi}-r^ke^{-ik\theta}|^2}} 
\,d\vphi$$

and
\begin{equation}
\label{eqn:formula-logu}
\log\left(\frac{u(re^{i\theta})}{r^k\sin k\theta}\right)=\log \frac{2k}{\pi}
+\log(1-r^{2k})+
\log g(z)
\end{equation}
where
$$g(re^{i\theta})=\int_{0}^{\pi/k} K(re^{i\theta}, e^{i\vphi})\sin(k\vphi) u(e^{i\vphi})\,d\vphi\ ,$$
with $$K(re^{i\theta},e^{i\vphi})=\frac{1}{|e^{ik\vphi}-r^ke^{ik\theta}|^2|e^{ik\vphi}-r^ke^{-ik\theta}|^2}$$

Now, $$\min \{K(z, \zeta) \big|\, z\in S_k, |z|<1/2, |\zeta|=1, 0\leq \arg\zeta\leq \pi/k\}$$ is a positive number $C_1$,
and
$$\max\{|\nabla_z K(z, \zeta)|\,\big| \, z\in S_k, |z|<1/2, |\zeta|=1, 0\leq\arg\zeta\leq\pi/k\}<C_2<\infty\ .$$  
Consequently, for $|z|<1/2$,
$$g(z)\geq C_1\int_{0}^{\pi/k}u(e^{i\vphi})\sin(k\vphi)\,d\vphi,$$
and $$|\nabla g(z)|\leq C_2\int_0^{\pi/k}u(e^{i\vphi})\sin(k\vphi)\, d\vphi\ .$$

So we get $|\nabla\log g|\leq C_2/C_1$ in $S_k\cap B_{1/2}$,
and then from~\eref{eqn:formula-logu}
$$|\nabla\log\frac{u}{r^k\sin k\theta}|\leq Ck\ $$
in $S_k\cap B_{1/2}$.
Finally we use reflection to get the stated result in the unit ball.


\vfill
\noindent Dan Mangoubi,\\ 
Einstein Institute of Mathematics,\\
Hebrew University, Givat Ram,\\
Jerusalem 91904,\\
Israel\\
\smallskip
\texttt{\small mangoubi@math.huji.ac.il}

%
\end{document}